\documentclass[11pt,reqno]{amsart}
\usepackage{amsmath,amssymb,amsthm}
\usepackage{graphicx}
\usepackage[all]{xypic}
\usepackage[utf8,nocaptions]{vietnam}
\usepackage{calrsfs}
\usepackage{hyperref}
\input xypic
\def\thmsection{section}
\def\thmchangesection{changesection}

\def\thmchangechapter{changechapter}
\def\thmchange{change}
\def\thmplain{plain}
\ifx\theoremnumberstyle\thmplain
  \theoremstyle{break-italic}
  \newtheorem{satz}{Satz}
\else
  \ifx\theoremnumberstyle\thmsection
    \theoremstyle{break-italic}
    \newtheorem{satz}{Satz}[section]
  \else
    \ifx\theoremnumberstyle\thmchange
      \swapnumbers
      \theoremstyle{break-italic}
      \newtheorem{satz}{Satz}
    \else
      \ifx\theoremnumberstyle\thmchangesection
         \swapnumbers
         \theoremstyle{break-italic}
         \newtheorem{satz}{Satz}[section]
      \else
        \ifx\theoremnumberstyle\thmchangechapter
           \swapnumbers
           \theoremstyle{break-italic}
           \newtheorem{satz}{Satz}[chapter]
        \else
          \ifx\theoremnumberstyle\thmsection
             \theoremstyle{break-italic}
             \newtheorem{satz}{Satz}[section]
          \else
            \theoremstyle{break-italic}
            \newtheorem{satz}{Satz}[section]
          \fi
        \fi
      \fi
    \fi
  \fi
\fi

\theoremstyle{break-italic}
\newtheorem{theorem}[satz]{Theorem}
\newtheorem{lemma}[satz]{Lemma}

\newtheorem{corollary}[satz]{Corollary}
\newtheorem{Proposition}[satz]{Proposition}

\newtheorem*{conjecture*}{Conjecture}

\theoremstyle{break-roman}
\newtheorem{definition}[satz]{Definition}

\newtheorem{example}[satz]{Example}

\newtheorem{remark}[satz]{Remark}

\newtheorem{conjecture}[satz]{Conjecture}

\theoremstyle{standard}

\newtheorem*{claim}{Claim}

\theoremstyle{varthm-roman}
\newtheorem*{varthm-roman}{}
\theoremstyle{varthm-italic}
\newtheorem*{varthm-italic}{}
\theoremstyle{varthm-roman-break}
\newtheorem*{varthm-roman-break}{}
\theoremstyle{varthm-italic-break}
\newtheorem*{varthm-italic-break}{}
\theoremstyle{varthm-roman-no-punctuation}
\newtheorem{varthm-roman-no-punctuation-numbered}[satz]{}
\theoremstyle{varthm-italic-no-punctuation}
\newtheorem{varthm-italic-no-punctuation-numbered}[satz]{}

\newenvironment{varthm-roman-numbered}[1]{
  \begin{varthm-roman-no-punctuation-numbered}
    \mbox{\rm\textbf{#1}}
  }{\end{varthm-roman-no-punctuation-numbered}}
\newenvironment{varthm-italic-numbered}[1]{
  \begin{varthm-italic-no-punctuation-numbered}
    \mbox{\rm\textbf{#1}}
  }{\end{varthm-italic-no-punctuation-numbered}}
\newenvironment{varthm-roman-break-numbered}[1]{
  \begin{varthm-roman-no-punctuation-numbered}
    \mbox{\rm\textbf{#1}\newline}
  }{\end{varthm-roman-no-punctuation-numbered}}
\newenvironment{varthm-italic-break-numbered}[1]{
  \begin{varthm-italic-no-punctuation-numbered}
    \mbox{\rm\textbf{#1}}\newline
  }{\end{varthm-italic-no-punctuation-numbered}}
\numberwithin{equation}{section}

\def\ex{\begin{example}
  }
  \def\eex{\end{example}}
\def\thr{\begin{theorem}}
\def\ethr{\end{theorem}}
\def\pro{\begin{Proposition}}
\def\epro{\end{Proposition}}
\def\coro{\begin{corollary}}
\def\ecoro{\end{corollary}}
\def\df{\begin{definition}}
\def\edf{\end{definition}}
\def\lm{\begin{lemma}}
\def\elm{\end{lemma}}
\def\pf{\begin{proof}}
\def\epf{\end{proof}}
\def\problem{\begin{problem}}
\def\eproblem{\end{problem}}

\def\ite{\begin{itemize}}
\def\hite{\end{itemize}}
\def\rem{\begin{remark}}
\def\erem{\end{remark}}

\def\cla{\begin{claim}}
\def\ecla{\end{claim}}
\def\conj{\begin{conjecture}}
\def\econj{\end{conjecture}}

\def\eex{{\accent"5E e}\kern-.385em\raise.2ex\hbox{\char'23}\kern-.08em}
\def\EES{{\accent"5E E}\kern-.5em\raise.8ex\hbox{\char'23 }}
\def\ow{o\kern-.42em\raise.82ex\hbox{
\vrule width .12em height .0ex depth .075ex \kern-0.16em \char'56}\kern-.07em}
\def\OW{O\kern-.460em\raise1.36ex\hbox{
\vrule width .13em height .0ex depth .075ex \kern-0.16em \char'56}\kern-.07em}

\def\dist{\mbox{dist}}
\begin{document}

\title[On Wielandt-Mirsky's conjecture for matrix polynomials]{On Wielandt-Mirsky's conjecture for matrix polynomials}
  
 \author{C\^{O}NG-TR\`{I}NH L\^{E}}
\address{Department of Mathematics, Quy Nhon University\\
 170 An Duong Vuong,  Quy Nhon, Binh Dinh, Vietnam}
\email{lecongtrinh@qnu.edu.vn}


\date{\today}



\begin{abstract} 
In matrix analysis, the  \textit{Wielandt-Mirsky conjecture} states that  
$$
\dist(\sigma(A), \sigma(B)) \leq  \|A-B\|, 
$$
for any normal matrices $ A, B \in \mathbb C^{n\times n}$ and any operator norm $\|\cdot \|$ on $C^{n\times n}$. Here $\dist(\sigma(A), \sigma(B))$ denotes the optimal matching  distance  between the spectra of the matrices $A$ and $ B$. It was proved by A.J. Holbrook (1992) that this conjecture is \textit{false } in general. However it is true for the Frobenius distance and the Frobenius norm (the \textit{Hoffman-Wielandt inequality}). The main aim of this paper is to study  the Hoffman-Wielandt inequality and some weaker versions of the Wielandt-Mirsky conjecture for matrix polynomials.
 \end{abstract}

\maketitle
\section{Introduction}
Let $\mathbb C^{n\times n}$ denote the set of all $n\times n$ matrices whose entries in $\mathbb C$. Let $A, B \in \mathbb C^{n \times n}$ be  complex matrices whose spectra are $\sigma(A)=\{\alpha_1,\cdots,\alpha_n\}$ and $\sigma(B)=\{\beta_1\cdots,\beta_n\}$, respectively. The \textit{optimal matching  distance} between $\sigma(A)$ and $\sigma(B)$ is defined by 
 $$ \dist(\sigma(A),\sigma(B)):=\min_\theta \max_{j=1,\cdots,n} |\alpha_j - \beta_{\theta(j)}|,$$
where the minimum is taken over all permutations $\theta$ on the set $\{1,\cdots,n\}$.

One of the interesting  conjectures in matrix analysis is   the  \textit{Wielandt-Mirsky conjecture} \cite{Bh2007}  which states that for any normal matrices $ A, B \in \mathbb C^{n\times n}$, 
\begin{equation}\label{conjecture}
\dist(\sigma(A), \sigma(B)) \leq  \|A-B\|, 
\end{equation} 
where $\|\cdot \|$ denotes the operator bound  norm. 

This conjecture  has been proved to be true in the following special cases (cf. \cite{Ho}):  
\begin{itemize}
\item[(1)] $A$ and $B$ are Hermitian (Weyl, 1912);
\item[(2)] $A$, $B$ and $A-B$ are normal (Bhatia, 1982);
\item[(3)] $A$ is Hermitian and $B$ is skew-Hermitian (Sunder, 1982);
\item[(4)] $A$ and $B$ are constant multiples of unitaries (Bhatia and Holbrook, 1985).
\end{itemize}

It has been proven by Holbrook (1992, \cite{Ho}) that this conjecture is \textit{false} in general. However, if we replace the optimal matching distance  $\dist(\sigma(A), \sigma(B)) $ by the \textit{Frobenius distance} between the spectra of the matrices $A$ and $B$,
$$\dist_F(\sigma(A), \sigma(B)):=\min_\theta \big[\sum_{j=1}^n |\alpha_j - \beta_{\theta(j)}|^2\big]^{\frac{1}{2}}, $$
then we have the following \textit{Hoffman-Wielandt  inequality} \cite{HoWie} 
\begin{equation}\label{Hoffman-Wielandt}
\dist_F(\sigma(A), \sigma(B))\leq |A-B|_F
\end{equation}
for any normal matrices $A$ and $B$. Here $|A-B|_F$ denotes the \textit{Frobenius norm}\footnotemark  
   \footnotetext{For a matrix $A =(a_{ij}) \in \mathbb C^{n\times n},$ the \textit{Frobenius norm} of $A$ is defined by
   $$|A|_F:=\sqrt{trace(AA^*)} = \big(\sum_{i,j=1}^n |a_{ij}|^2\big)^{\frac{1}{2}}.$$ It is easy to see that $|A|_F^2=|AA^*|_F$.}  of the matrix $A-B$.
 
It is clear that 
$$\dist(\sigma(A),\sigma(B)) \leq \dist_F(\sigma(A),\sigma(B)) \leq \sqrt{n}\cdot \dist(\sigma(A),\sigma(B)). $$
Therefore it follows from the Hoffman-Wielandt inequality that \textit{the Wielandt-Mirsky  conjecture is true for the Frobenius norm}.

A weaker version of the Wielandt-Mirsky conjecture was proved by R. Bhatia, C. Davis and A. Mcintosh (1983, \cite{BDM1983}) that there exists a universal constant $c$ such that for any normal matrices $A, B \in \mathbb C^{n \times n}$ we have 
\begin{equation} \label{universal-constant}
\dist(\sigma(A), \sigma(B))\leq c \|A-B\|. 
\end{equation}

If we don't require the universality of the constant $c$ in the above inequality, for any normal matrix $A\in \mathbb C^{n\times n}$ and for any $B\in \mathbb C^{n\times n}$, we have (cf. \cite[p. 4]{Bh2007})
\begin{equation} \label{Anormal-Barbitrary}
\dist(\sigma(A),\sigma(B))\leq (2n-1)\|A-B\|.
\end{equation}

If $A$ is Hermitian and $B$ is arbitrary, we have the following inequality due to  W. Kahan (\cite[p.166]{K1967}:
\begin{equation} \label{Eq-Kahan}
\dist(\sigma({A}), \sigma({B})) \leq (\gamma_{n}+2)\|{A}-{B}\|,
\end{equation}
where $\gamma_n$ is a constant depending on the size $n$ of the matrices.

For a \textit{matrix polynomial} we mean the matrix-valued function of a complex variable of the form 
\begin{equation}\label{mp}
P(z)= A_m z^m + \cdots + A_1 z + A_0, 
\end{equation}
where $A_i\in \mathbb C^{n\times n}$ for all $i=0,\cdots,m$. If $A_m\not =0$, $P(z)$ is called a \textit{matrix polynomial of degree $m$}. When $A_m=I$, the identity matrix in $\mathbb C^{n \times n}$, the matrix polynomial $P(z)$ is called a \textit{monic}.

A number $\lambda \in \mathbb C$ is called an \textit{eigenvalue} of the matrix polynomial $P(z)$, if there exists a nonzero vector $x\in \mathbb C^n$ such that $P(\lambda)x=0$. Then the vector $x$ is called, as usual,  an \textit{eigenvector} associated  to the eigenvalue $\lambda$. Note that  each  eigenvalue of $P(z)$ is a root  of  the \textit{characteristic polynomial} $\det(P(z))$. 

For an $(m+1)$-tuple  $\bold{A}=(A_0,\cdots,A_m)$   of matrices $A_i\in \mathbb C^{n\times n}$,  the matrix polynomial 
$$ P_\bold{A}(z):= A_m z^m + \cdots + A_1 z + A_0$$
is called the \textit{matrix polynomial associated to $\bold{A}$}.  

The \textit{spectrum} of   the matrix polynomial $P_{\bold{A}}(z)$ is defined by 
$$\sigma(\bold{A}):=\sigma\big(P_{\bold{A}}(z)\big)=\{\lambda \in \mathbb C | \det(P_\bold{A}(\lambda))=0\}, $$
which is the set of all its eigenvalues. We should observe that for a matrix $A\in \mathbb C^{n \times n}$, its usual spectrum $\sigma(A)$ is actually the spectrum of the monic matrix polynomial $I z - A$.
Interested readers may refer to the book of I. Gohberg, P. Lancaster and L. Rodman \cite{GLR1982} for the theory of matrix polynomials and applications.

The main goal of this paper  is to  give some versions of the  Wielandt-Mirsky conjecture  for matrix polynomials. 

The paper is organized as follows. In Section \ref{section2} we  establish some Wielandt's  inequalities for   matrix polynomials. In Section \ref{section3} we give some weaker versions of the Wielandt-Mirsky  conjecture for monic matrix polynomials and for matrix polynomials whose leading coefficients are non-singular.

{\bf Notation and conventions.}   Throughout this paper,  by a positive integer $p$ we mean  $p\geq 1$ or $p=\infty$.\\
  For a matrix $A=(a_{ij}) \in \mathbb C^{n\times n}$,   a positive integer  $p$, and  a vector $p$-norm $|\cdot|_p$ on $\mathbb C^n$, the \textit{matrix $p$-norm} of $A$ is defined by 
$$ 
|A|_p:=\begin{cases}
\Big(\displaystyle\sum_{i,j=1}^n |a_{ij}|^p\Big)^{\frac{1}{p}} & (1\leq p < \infty)\\
\displaystyle\max_{i,j=1,\cdots,n} |a_{ij}| & (p=\infty).
\end{cases}
$$
In particular, $|A|_2=|A|_F$, the Frobenius norm.

The \textit{operator $p$-norm} of $A$ is defined by 
$$\|A\|_p:=\max\{ |Ax|_p: |x|_p =1\}. $$
Note that 
$$\|A\|_1:=\max_{j=1,\cdots,n} \sum_{i=1}^n a_{ij},  $$
$$\|A\|_\infty:=\max_{i=1,\cdots,n} \sum_{j=1}^n a_{ij}.$$
There are many relations between operator and matrix $p$-norms. Interested readers may refer to the paper of A. Tonge  \cite{T2000} and the references therein for more details.

\section{Some Wielandt's  inequalities for   matrix polynomials}\label{section2}
In the first part of this section we give some versions of the Hoffman-Wielandt  inequality (\ref{Hoffman-Wielandt}) for monic matrix polynomials. 

For a monic matrix polynomial 
$P_{\bold{A}}(z)= I\cdot  z^m + A_{m-1}z^{m-1}+ \cdots + A_1 z + A_0 $ with  $A_i\in \mathbb C^{n\times n}$,  the $(mn\times mn)$-matrix 
$$ C_{\bold{A}}:=\left[\begin{array}{ccccc}
0 & I & 0 & \cdots & 0\\
0 & 0 & I & \cdots & 0\\
\vdots & \vdots & \vdots &   & \vdots \\
0 & 0 & 0 & \cdots & I\\
-A_0 & -A_1 & -A_2 & \cdots & -A_{m-1}
\end{array}\right]
$$ 
is called the \textit{companion matrix} of the matrix polynomial $P_{\bold{A}}(z)$ or of the tuple $(A_0,\cdots,A_{m-1},I)$. 

 Note that the spectrum $\sigma(\bold{A})$ of $P_{\bold{A}}(z)$ coincides to the spectrum $\sigma(C_\bold{A})$ of  $ C_{\bold{A}}$ (cf. \cite{GLR1982}).

For  two $(m+1)$-tuples $\bold{A}=(A_0,\cdots,A_{m-1},I)$ and $\bold{\bar{A}}=(\bar{A}_0,\cdots,\bar{A}_{m-1},I)$, the relation between the operator norms   of their difference and those of their companion matrices is given in the following key lemma.  
 \lm  \label{norm-companion} Let  $\bold{A}=(A_0,\cdots,A_{m-1},I)$ and $\bold{\bar{A}}=(\bar{A}_0,\cdots,\bar{A}_{m-1},I)$ be $(m+1)$-tuples.   Then for any integer  $p>0$, we have 
 \begin{itemize}
 \item[(1)] $ |C_\bold{A} - C_{\bold{\bar{A}}}|_p = |\bold{A}-\bold{\bar{A}}|_p=
  \begin{cases}
  \big(\sum_{i=0}^{m-1} |A_i-\bar{A}_i|_p^p\big)^{\frac{1}{p}} & (1\leq p < \infty)\\
  \displaystyle\max_{i=0,\cdots,m} |A_i|_\infty  & (p=\infty).
  \end{cases}$ 
 \item[(2)] $ \|C_\bold{A} - C_{\bold{\bar{A}}}\|_p = \|\bold{A}-\bold{\bar{A}}\|_p\leq \displaystyle\sum_{i=0}^{m-1} \|A_i-\bar{A}_i\|_p$.
 \item[(3)] $ \|C_\bold{A} - C_{\bold{\bar{A}}}\|_1 =   \displaystyle\max_{i=0,\cdots,m-1} \|A_i-\bar{A}_i\|_1.$
  \end{itemize}
   \elm
  
 \pf We have the following expression of the difference of companion matrices  
 $$ C_\bold{A} - C_{\bold{\bar{A}}} = \left[\begin{array}{ccccc}
0 & 0 & 0 & \cdots & 0\\
0 & 0 & 0 & \cdots & 0\\
\vdots & \vdots & \vdots &   & \vdots \\
0 & 0 & 0 & \cdots & 0\\
\bar{A}_0-A_0 & \bar{A}_1-A_1 & \bar{A}_2-A_2 & \cdots & \bar{A}_{m-1}-A_{m-1}
\end{array}\right].
 $$
This implies  that the  matrix (resp. operator) norm of $C_\bold{A} - C_{\bold{\bar{A}}}$ is the same of that of the $m$-tuple $(\bar{A}_0-A_0, \ldots, \bar{A}_{m-1}-A_{m-1})$, i.e. we have  the first equalities  in (1) and (2). \\
 The second equality in (2) follows from the subadditivity of the operator $p$-norm. On the other hand, for an $(m+1)$-tuple  $\bold{A}=(A_0,\cdots,A_m)$ of matrices in $\mathbb C^{n\times n}$, by a direct computation, we have   
\begin{equation}\label{matrixnorm-1}
  |\bold{A}|_p:= \begin{cases}
  \big(\sum_{i=0}^m |A_i|_p^p\big)^{\frac{1}{p}} & (1\leq p < \infty)\\
  \displaystyle\max_{i=0,\cdots,m} |A_i|_\infty  & (p=\infty),
  \end{cases}
\end{equation}
thus we get the second equality of (1). Moreover, we have 
\begin{equation}\label{oper1norm}
 \|\bold{A}\|_1 = \max_{i=0,\cdots,m} \|A_i\|_1.
 \end{equation}
 Thus, we get (3). 
 \epf 
 As a consequence of Lemma \ref{norm-companion},  we  obtain the following   Hoffman-Wielandt  inequality for matrix polynomials.
 
 \pro  \label{thr-Hoffman-Wielandt-matrix} Let  $P_{\bold{A}}(z)= I\cdot  z^m + A_{m-1}z^{m-1}+ \cdots + A_1 z + A_0$ and $P_{\bold{\bar{A}}}(z)= I\cdot  z^m + \bar{A}_{m-1}z^{m-1}+ \cdots + \bar{A}_1 z + \bar{A}_0$ be monic matrix polynomials  whose corresponding companion matrices  $C_\bold{A}$ and $C_\bold{\bar{A}}$ are normal. Then we have 
$$\dist_F(\sigma(\bold{A}), \sigma(\bold{\bar{A}})) \leq |\bold{A}-\bold{\bar{A}}|_F. $$
\epro 
\pf Applying the  Hoffman-Wielandt inequality  (\ref{Hoffman-Wielandt})   for two normal  matrices $C_\bold{A}$ and    $C_{\bold{\bar{A}}}$ we get
$$\dist_F(\sigma(C_\bold{A}), \sigma(C_{\bold{\bar{A}}})) \leq |C_\bold{A}-C_{\bold{\bar{A}}}|_F. $$
 Then the theorem  follows from Lemma  \ref{norm-companion} (applying for $p=2$) with the observation that  $\sigma(\bold{A})=\sigma(C_\bold{A})$ and $\sigma(\bold{\bar{A}})=\sigma(C_\bold{\bar{A}})$. 
\epf
We should observe that 
\begin{equation}\label{companion-normal}
C_\bold{A} \mbox{ is normal if and only if } A_0 \mbox{ is unitary and } A_1=\ldots=A_{m-1}=0.
\end{equation}

 Therefore, Theorem \ref{thr-Hoffman-Wielandt-matrix}  yields  the following consequence.
 \coro \label{coro-Hoffman-Wielandt}
 Let   $P_{\bold{A}}(z)= I\cdot  z^m + A_0$ and $P_{\bold{\bar{A}}}(z)= I\cdot  z^m +  \bar{A}_0$ be momic matrix polynomials  with   $A_0$ and  $\bar{A}_0$ unitary. Then  we have   
$$\dist_F(\sigma(\bold{A}), \sigma(\bold{\bar{A}})) \leq |A_0-\bar{A}_0|_F. $$
\ecoro

One more interesting inequality was established by Wielandt for scalar  matrices which is stated as follows.
\begin{theorem}[{cf. \cite[Theorem 1.45]{Hi}}]
Let $A\in \mathbb C^{n \times n}$ be a Hermitian matrix such that for some numbers $a,b>0$ the inequalities $bI \leq A \leq aI$ hold.   Then for any orthogonal unit vectors $x,y \in \mathbb C^n$, the inequality  
\begin{equation}\label{Wielandt}
 |x^*Ay|^2 \leq \big(\dfrac{a-b}{a+b}\big)^2(x^*Ax)\cdot (y^*Ay)
\end{equation}
holds.
\end{theorem}

In the following we establish a version of this inequality for matrix polynomials.
\begin{theorem}
Let $P_A(\lambda)=\displaystyle\sum_{i=0}^d A_i \lambda^i$ be a matrix polynomial whose coefficient matrices satisfy 
$$bI\leq A_i\leq aI$$
 for some numbers $a,b >0$ and for all  $i=0,\ldots,d$. Then for any orthogonal unit  vectors $x,y \in \mathbb C^n$,
$$|x^*P_A(\lambda)y|^2 \leq \big(\dfrac{a-b}{a+b}\big)^2 (x^*P_A(|\lambda|)x) \cdot  (y^*P_A(|\lambda|)y).$$
\end{theorem}
\begin{proof}
It follows from (\ref{Wielandt}) that for each $i=0,\ldots,d$,  
 $$|x^*A_iy|\leq \big(\dfrac{a-b}{a+b}\big)(x^*A_ix)^{\frac{1}{2}}\cdot (y^*A_iy)^{\frac{1}{2}}.$$
 Then
 \begin{align*}
|x^*P_A(\lambda)y|^2 &= \big(\big|\sum_{i=0}^d (x^*A_iy)\lambda^i \big|\big)^2 \leq  \big( \sum_{i=0}^d |x^*A_iy|\cdot |\lambda|^i \big)^2\\
& \leq \big(\dfrac{a-b}{a+b}\big)^2 \big( \sum_{i=0}^d (x^*A_ix)^{\frac{1}{2}} |\lambda|^{\frac{i}{2}}\cdot (y^*A_iy)^{\frac{1}{2}}\cdot |\lambda|^{\frac{i}{2}} \big)^2\\
& \leq \big(\dfrac{a-b}{a+b}\big)^2 \big(\sum_{i=0}^d (x^*A_ix)  |\lambda|^{i}\big)  \big(\sum_{i=0}^d (y^*A_iy)  |\lambda|^{i}\big)\\
&= \big(\dfrac{a-b}{a+b}\big)^2 \big(x^*(\sum_{i=0}^d A_i  |\lambda|^{i})x\big)\big(y^*(\sum_{i=0}^d A_i  |\lambda|^{i})y\big)\\
&= \big(\dfrac{a-b}{a+b}\big)^2\big(x^*P_A(|\lambda|)x\big)\cdot \big(y^*P_A(|\lambda|)y\big),
 \end{align*}
where the last inequality follows from the classical Cauchy-Schwarz inequality. The proof is complete.
\end{proof}

\section{ Some weaker versions of the Wielandt-Mirsky conjecture for matrix polynomials } \label{section3}

In this section we give some estimations for  the optimal matching distance between the spectra of  matrix polynomials. 

\pro  \label{thr-distance-1} There exists a constant $c>0$ such that for every positive integer $p$ and   monic matrix polynomials $P_{\bold{A}}(z)= I\cdot  z^m + A_{m-1}z^{m-1}+ \cdots + A_1 z + A_0$ and $P_{\bold{\bar{A}}}(z)= I\cdot  z^m + \bar{A}_{m-1}z^{m-1}+ \cdots + \bar{A}_1 z + \bar{A}_0$ whose corresponding companion matrices  $C_\bold{A}$ and $C_\bold{\bar{A}}$ are normal we have 
$$\dist(\sigma(\bold{A}), \sigma(\bold{\bar{A}})) \leq c\|\bold{A}-\bold{\bar{A}}\|_p \leq c\displaystyle\sum_{i=0}^{m-1} \|A_i-\bar{A}_i\|_p.  $$
\epro 
\pf  It follows from the result of R. Bhatia, C. Davis and A. Mcintosh (see the inequality (\ref{universal-constant}) that  there exists a constant $c>0$ such that for every monic matrix polynomials $P_{\bold{A}}(z)$ and $P_{\bold{\bar{A}}}(z)$ whose corresponding companion matrices  $C_\bold{A}$ and $C_\bold{\bar{A}}$ are normal, we have
$$\dist(\sigma(C_\bold{A}), \sigma(C_{\bold{\bar{A}}})) \leq c\|C_\bold{A}-C_{\bold{\bar{A}}}\|_p. $$
  Then the theorem  follows from Lemma  \ref{norm-companion} (2) with the observation that  $\sigma(\bold{A})=\sigma(C_\bold{A})$ and $\sigma(\bold{\bar{A}})=\sigma(C_\bold{\bar{A}})$. 
\epf 
  Using again the observation (\ref{companion-normal}) we obtain the following consequence.
\coro \label{coro-distance-1-1}
 There exists a constant $c>0$ such that for every positive integer $p$ and monic matrix polynomials $P_{\bold{A}}(z)= I\cdot  z^m + A_0$ and $P_{\bold{\bar{A}}}(z)= I\cdot  z^m +  \bar{A}_0$ with   $A_0$ and  $\bar{A}_0$ unitary, we have   
$$\dist(\sigma(\bold{A}), \sigma(\bold{\bar{A}})) \leq c\|A_0-\bar{A}_0\|_p. $$
\ecoro
Similarly, applying the inequality (\ref{Anormal-Barbitrary}) and Lemma  \ref{norm-companion} (2) with the observation (\ref{companion-normal}) we obtain the following results.
 \coro \label{coro-distance-1-2}
 Let  $P_{\bold{A}}(z)= I\cdot  z^m + A_0$ and $P_{\bold{\bar{A}}}(z)= I\cdot  z^m +  \bar{A}_0$ be matrix polynomials with   $A_0$  and $\bar{A}_0$  unitary. Then  for every positive integer $p$ we have  
$$\dist(\sigma(\bold{A}), \sigma(\bold{\bar{A}})) \leq (2mn-1)\|A_0-\bar{A}_0\|_p. $$
\ecoro
 
A similar version of  (\ref{Eq-Kahan}) for monic matrix polynomials is given as follows, whose proof is similar to that of Proposition \ref{thr-Hoffman-Wielandt-matrix}. 

\pro \label{thr-distance-2} Let $P_{\bold{A}}(z)= I\cdot  z^m + A_{m-1}z^{m-1}+ \cdots + A_1 z + A_0$ and $P_{\bold{\bar{A}}}(z)= I\cdot  z^m + \bar{A}_{m-1}z^{m-1}+ \cdots + \bar{A}_1 z + \bar{A}_0$ be monic matrix polynomials. Assume that $C_\bold{A}$ is Hermitian. Then for every positive integer $p$ we have  
$$\dist(\sigma(\bold{A}), \sigma(\bold{\bar{A}})) \leq (\gamma_{m,n}+2)\|\bold{A}-\bold{\bar{A}}\|_p \leq (\gamma_{m,n}+2)\displaystyle\sum_{i=0}^{m-1} \|A_i-\bar{A}_i\|_p, $$
where $\gamma_{m,n}$ is a constant depending on $m$ and $n$. 
\epro 

\rem \rm The constant $\gamma_{m,n}$ have the following properties (cf. \cite[p. 3]{Bh2007}):
\begin{itemize}
\item[(1)] $\dfrac{2}{\pi} \ln (mn) - 0(1) \leq \gamma_{m,n} \leq \log_2 (mn) + 0.038.$
\item[(2)] A. Pokrzywa (1981, \cite{P1981}) proved that 
$$\gamma_{m,n}=\dfrac{2}{mn} \sum_{j=1}^{[mn/2]} \cot \dfrac{2j-1}{2mn}\pi.$$
\end{itemize}
\erem 

One of the condition for the companion matrix $C_\bold{A}$ to be  Hermitian is given as follows. 
\coro \label{coro-distance-2-1}  Let  $P_{\bold{A}}(z)= I\cdot  z^m + A_0$ and $P_{\bold{\bar{A}}}(z)= I\cdot  z^m +  \bar{A}_0$ be matrix polynomials with   $A_0$   unitary. Assume that $P_{\bold{A}}(z)$ has only real eigenvalues.   Then for every positive integer $p$ we have    
$$\dist(\sigma(\bold{A}), \sigma(\bold{\bar{A}})) \leq (\gamma_{m,n}+2)\|A_0-\bar{A}_0\|_p. $$
\ecoro 
\pf
It is well-known that a normal matrix $A\in \mathbb C^{n\times n}$ is Hermitian if and only if it has only real eigenvalues. Therefore, the normal matrix $C_\bold{A}$ is Hermitian if and only if it, whence the matrix polynomial $P_\bold{A}(z)$, has only real eigenvalues. Note also that in this case,  $C_\bold{A}$ is normal if and only if $A_0$ is unitary. Then the result follows from Proposition  \ref{thr-distance-2}.
\epf
\rem \rm 
There are some characterization for matrix polynomials  to  have only real eigenvalues. These kinds of matrix polynomials are sometime called \textit{weakly hyperbolic}.  The readers can refer to the work of  M. Al-Ammari and F. Tisseur (2012, \cite[Theorem 3.4]{AT2012}) and the references therein for more details characterization.
\erem 

In the following we will establish an estimation for the optimal matching distance between spectra of two arbitrary monic matrix polynomials. For the proof, we need the following  estimation given by Bhatia and Friedland (1981), Elsner (1982, 1985).  

\lm[{\cite[Theorem 20.4]{Bh2007}}] \label{bhatia-arbitrary} Let $A$ and $B$ be any $k \times k$-matrices.  Then the optimal matching distance between their eigenvalues are bounded as 
$$\dist(\sigma(A), \sigma(B)) \leq  c(k) \cdot  k^{\frac{1}{k}} \cdot  (2M)^{1-\frac{1}{k}}\cdot \|A-B\|^{\frac{1}{k}}, $$
where $M=\max\{\|A\|, \|B\|\}$, $\|\cdot \|$ any operator norm, and $c(k)=k $ or $k-1$ according to whether $k$ is odd or even.  
\elm   

\thr Let $N$ be  any positive  number and $p$ any positive integer. Let $P_{\bold{A}}(z)= I\cdot  z^m + A_{m-1}z^{m-1}+ \cdots + A_1 z + A_0$ and $P_{\bold{\bar{A}}}(z)= I\cdot  z^m + \bar{A}_{m-1}z^{m-1}+ \cdots + \bar{A}_1 z + \bar{A}_0$ be monic matrix polynomials such that  $|\bold{A}|_p \leq N$ and $|\bold{\bar{A}}|_p\leq N$. Then    there exists a constant $c$ such that 
$$\dist(\sigma(\bold{A}), \sigma(\bold{\bar{A}})) \leq c \|\bold{A}-\bold{\bar{A}}\|_p^{\frac{1}{mn}}. $$
\ethr

\pf 
Applying Lemma \ref{bhatia-arbitrary} for the companion matrices $C_\bold{A}$ and $C_\bold{\bar{A}}$ with the operator  norm $\|\cdot \|_p$ we obtain 
$$\dist(\sigma(C_\bold{A}), \sigma(C_\bold{\bar{A}})) \leq  c(mn) \cdot  (mn)^{\frac{1}{mn}} \cdot  (2M)^{1-\frac{1}{mn}}\cdot \|C_\bold{A}-C_\bold{\bar{A}}\|_p^{\frac{1}{mn}}, $$
where $M=\max\{\|C_\bold{A}\|_p, \|C_\bold{\bar{A}}\|_p\}$. Then it follows from  Lemma  \ref{norm-companion} (2) and  the equalities   $\sigma(\bold{A})=\sigma(C_\bold{A})$ and $\sigma(\bold{\bar{A}})=\sigma(C_\bold{\bar{A}})$ that 
$$\dist(\sigma(\bold{A}), \sigma(\bold{\bar{A}})) \leq  c(mn) \cdot  (mn)^{\frac{1}{mn}} \cdot  (2M)^{1-\frac{1}{mn}}\cdot \|\bold{A}-\bold{\bar{A}}\|_p^{\frac{1}{mn}}. $$
Now we need to estimate $\|C_\bold{A}\|_p$ and $ \|C_\bold{\bar{A}}\|_p$.   By a comparison of the operator $p$-norm $\|C_\bold{A}\|_p$ and the matrix $p$-norm $|C_\bold{A}|_p$ given by M. Gohberg \cite{G1987} (see also in \cite{T2000}, we have 
\begin{equation}\label{compare-pnorm}
\|C_\bold{A}\|_p \leq (mn)^{\max\{1/p'-1/p,0\}} |C_\bold{A}|_p.
\end{equation}
It is easy to see that for $1\leq p < \infty$,  
\begin{equation}\label{equ-compare-norm1}
|C_\bold{A}|_p^p = \sum_{i=0}^{m-1}|A_i|_p^p + (m-1)n = |\bold{A}|_p^p + (m-1)n \leq N^p + (m-1)n,
\end{equation}
and for $p=\infty$, 
\begin{equation}\label{equ-compare-norm2}
|C_\bold{A}|_\infty  =  \max\{|\bold{A}|_\infty, 1\}\leq \max\{N,1\}.
\end{equation}
 It follows from the inequalities (\ref{compare-pnorm}),  (\ref{equ-compare-norm1}) and (\ref{equ-compare-norm2}) that  
$$\|C_\bold{A}\|_p \leq M':=\begin{cases}
(mn)^{\max\{1/p'-1/p,0\}} \big(N^p + (m-1)n\big)^{\frac{1}{p}} & (1\leq p < \infty)\\
mn\max\{N,1\} & (p=\infty).
\end{cases} $$
Similarly,
$$\|C_\bold{\bar{A}}\|_p \leq M'.$$
Then the number  $$c:= c(mn) \cdot  (mn)^{\frac{1}{mn}} \cdot  (2M')^{1-\frac{1}{mn}}$$
satisfies the requirement. 
\epf 

For  matrix polynomials which are not necessarily monic, we have the following estimation, only for operator $1$-norm. 

\thr \label{thr1} Let $\bold{\bar{A}}=(\bar{A}_0,\cdots, \bar{A}_m)$ be a fixed $(m+1)$-tuple  such that $\bar{A}_m$ non-singular. Then there exists a constant $c>0$ such that for every pair of   matrix polynomials  $P_{\bold{A}}(z)= A_m\cdot  z^m + A_{m-1}z^{m-1}+ \cdots + A_1 z + A_0$ and $P_{\bold{{A}'}}(z)= A'_m\cdot  z^m +  {A}'_{m-1}z^{m-1}+ \cdots + {A}'_1 z +  {A}'_0$ with   $\|\bold{A}-\bold{\bar{A}}\|_1 < \dfrac{1}{\|\bar{A}^{-1}_m\|_1}$ and $\|\bold{A'}-\bold{\bar{A}}\|_1 < \dfrac{1}{\|\bar{A}^{-1}_m\|_1}$ we have 
$$\dist(\sigma(\bold{A}), \sigma(\bold{{A}'})) < c  \|\bold{A}-\bold{A'}\|_1^{\frac{1}{mn}}. $$
\ethr 

\pf It follows from the sub-multiplicative property of operator $1$-norm and the formula (\ref{oper1norm}) that 
$$ \|\bar{A}^{-1}_m(A_m-\bar{A}_m)\|_1 \leq \|\bar{A}^{-1}_m \|_1\cdot \|A_m-\bar{A}_m\|_1\leq  \|\bar{A}^{-1}_m \|_1 \cdot \|\bold{A}-\bold{\bar{A}}\|_1 < 1. $$
Then  $A_m$ is  also  non-singular (cf. \cite[Theorem 2.3.4]{GL2013}), thus each element of $\sigma(\bold{ A})$ is finite.  Similarly, each element of $\sigma(\bold{ A'})$ is also finite.

Since $\sigma(\bold{A'})$ is the set of roots of the polynomial $\det(P_\bold{A'}(z))\in \mathbb C[z]$ whose degree is $mn$, it is easy to see that 
$$\dist(x, \sigma(\bold{A'}))^{mn}\leq  |\det(P_{\bold{A'}}(x))|, \mbox{ for all } x\in \mathbb C.$$
In particular, for any    $\lambda \in \sigma(\bold{{A}})$, we have
\begin{equation} \label{equ6}
 \dist(\lambda, \sigma(\bold{A'}))^{mn} \leq |\det(P_{\bold{A'}}(\lambda))|.
 \end{equation}
 
Since 
$\det(P_{ \bold{A}}(\lambda))=0,$
we have 
\begin{align} \label{equ7}
|\det(P_{\bold{A'}}(\lambda))| & =  |\det(P_{\bold{A'}}(\lambda)) - \det(P_{\bold{A}}(\lambda))|.
\end{align}

The following inequality is useful for the later estimation (cf. \cite[2007, p.107]{Bh2007}): For any matrices $A,B \in \mathbb C^{n\times n}$ and for any integer $p>0$,  we have 
\begin{equation}\label{det-ineq}
|\det(A) - \det(B)| \leq n \max\{\|A\|_p, \|B\|_p\}^{n-1}\cdot \|A-B\|_p.
\end{equation} 
Applying the inequality (\ref{det-ineq}), we get
\begin{equation} \label{equ8}
|\det(P_{\bold{A'}}(\lambda)) - \det(P_{ \bold{A}}(\lambda))| \leq n \max\{  \|P_{ \bold{A'}}(\lambda)\|_1, \|P_{\bold{A}}(\lambda)\|_1\}^{n-1}\cdot \|P_{\bold{A}}(\lambda) - P_{ \bold{A'}}(\lambda)\|_1. 
\end{equation}
Using again the sub-multiplicativity  of operator $1$-norm and the formula (\ref{oper1norm}), we have
$$\|P_{\bold{A}}(\lambda) - P_{ \bold{A'}}(\lambda)\|_1 = \|\sum_{i=0}^m (A_i -A'_i)\lambda^i \|_1 \leq \sum_{i=0}^m |\lambda|^i \cdot \|\bold{A}-\bold{A'}\|_1.$$
It follows from a matrix version of Cauchy theorem (cf. \cite[Theorem 3.4]{DLN2017}) that for   $\lambda \in \sigma(\bold{{A}})$, we have 
$$ |\lambda| < 1 + \|{A}_m^{-1}\|_1 \cdot \max\{\|{A}_i\|_1, i=0,\cdots,m-1\} \leq 1+\|{A}_m^{-1}\|_1\cdot \|\bold{{A}}\|_1.$$
On the other hand, by the sub-additivity of operator norm, we have 
$$\|\bold{A}\|_1\leq \|\bold{A}-\bold{\bar{A}}\|_1 + \|\bold{\bar{A}}\|_1 <  \dfrac{1}{\|\bar{A}^{-1}_m\|_1} + \|\bold{\bar{A}}\|_1.$$ 
Hence for each $\lambda \in \sigma(\bold{A})$ we have 
$$|\lambda| < 2 +  \|\bar{A}_m^{-1}\|_1\cdot \|\bold{\bar{A}}\|_1.$$
Then 
$$\sum_{i=0}^m |\lambda|^i < L:=\dfrac{(2 +  \|\bar{A}_m^{-1}\|_1\cdot \|\bold{\bar{A}}\|_1)^{m+1}-1}{1+ \|\bar{A}_m^{-1}\|_1\cdot \|\bold{\bar{A}}\|_1}.$$
 This yields
\begin{equation}\label{equ9}
\|P_{\bold{A}}(\lambda) - P_{ \bold{A'}}(\lambda)\|_1 < L\cdot \|\bold{A}-\bold{A'}\|_1.
\end{equation}
By a similar  estimation, we get 
\begin{align}
\|P_{ \bold{A}}(\lambda)\|_1 & < L \cdot \|\bold{A}\|_1\leq L\cdot \big(\dfrac{1}{\|\bar{A}^{-1}_m\|_1} + \|\bold{\bar{A}}\|_1\big),\label{equ10}\\
\|P_{\bold{A'}}(\lambda)\|_1 & <  L\cdot \|\bold{A'}\|_1\leq L\cdot \big(\dfrac{1}{\|\bar{A}^{-1}_m\|_1} + \|\bold{\bar{A}}\|_1\big).\label{equ11}
\end{align}
It follows from the inequalities (\ref{equ6}), (\ref{equ7}), (\ref{equ8}), (\ref{equ9}), (\ref{equ10}) and  (\ref{equ11})  that  
  $$ \dist(\lambda, \sigma(\bold{A'}))  < c \|\bold{A}-\bold{A'}\|_1^{\frac{1}{mn}},$$ 
 where  $$ c:=\Big(n\cdot L^{n-1}\cdot \big(\dfrac{1}{\|\bar{A}^{-1}_m\|_1} + \|\bold{\bar{A}}\|_1\big)^{n-1}\Big)^{\frac{1}{mn}}.$$
  Similarly, for every $\lambda' \in \sigma(\bold{A'}),$ we have also 
 $$ \dist(\lambda, \sigma(\bold{A}))  < c \|\bold{A}-\bold{A'}\|_1^{\frac{1}{mn}}.$$
 It follows that 
 $$ \dist(\sigma(\bold{A}),\sigma(\bold{A'}))  < c \|\bold{A}-\bold{A'}\|_1^{\frac{1}{mn}}.$$
\epf 

\section*{Acknowledgements} The author would like to thank Dr. Minh Toan HO for his suggestion to establish a Wielandt's inequality for matrix polynomials in Section 2. He would also like to express his sincere gratitude to the anonymous referees for their useful comments and suggestions to improve the results in the original version of this paper. The final version of this paper was finished during the visit of the  author   at the Vietnam Institute for Advanced Study in Mathematics (VIASM). He thanks VIASM for financial support and hospitality.


\begin{thebibliography}{99}
\bibitem{AT2012}     Al-Ammari M, Tisseur F.  { Hermitian matrix polymomials with real eigenvalues of definite type. Part I. Classification}.  Linear Algebra Appl. 2012; \textbf{436}(10): 3954–3973.
\bibitem{Bh2007}     Bhatia R.  { Perturbation bounds for matrix eigenvalues}. Longman, Harlow; 1987; 2nd ed., SIAM, Philadelphia; 2007. 
\bibitem{BDM1983}    Bhatia R,  Davis C,  Mcintosh A.  { Perturbation of spectral subspaces and solution of linear operator equations}.   Linear algebra and its applications  1983; \textbf{52}:  45-67.
\bibitem{DLN2017}    Du T.H.B,  Le C.-T., Nguyen T.D.  { On the location of eigenvalues of matrix polynomials}.   Preprint, 2017. 


\bibitem{GLR1982}   Gohberg I,  Lancaster P, Rodman L.  {Matrix Polynomials}.  Academic Press, New York; 1982. 
\bibitem{G1987}    Goldberg M. {Equivalence constants for lp norms of matrices}.   Linear and Multilinear Algebra 1987; \textbf{21}:   173–179.
 \bibitem{GL2013}   Golub G.H,    Van Loan C.F.   Matrix Computations,  4th Edition.  Johns Hopkins University Press, Baltimore;  2013.
  \bibitem{Hi}   Hiai F,  Petz D. Introduction to matrix analysis and applications. Springer; 2014.
\bibitem{HoWie}   Hoffman A.J,   Wielandt H.W.  The variation of the spectrum of a normal matrix. Duke Math. J. 1953; \textbf{20}:  37-39.
\bibitem{Ho}    Holbrook A.J.      {Spectral  variation  of  normal  matrices}.   Linear  Algebra  Appl. 1992; \textbf{174}: 131-144. 
\bibitem{K1967}    Kahan W.  { Inclusion theorems for clusters of eigenvalues of Hermitian  matrices}.  Technical Report, Computer Science Department, University  of Toronto; 1967. 
\bibitem{P1981} { Pokrzywa A}.  { Spectra of operators with fixed imaginary parts}. Proc. Amer. Math. Soc. 1981; \textbf{81}, No. 3: 359-364.
\bibitem{T2000} { Tonge A}.  { Equivalence constants for matrix norms: a problem of Goldberg}. Linear Alg. Appl. 2000; \textbf{306}:  1-13.



\end{thebibliography}
\end{document}